\documentclass[12pt]{amsart}
\usepackage[utf8]{inputenc}
\usepackage{tikz} 
\usepackage{tikz-cd}
\usepackage[all]{xy}
\usepackage{amsthm} 
\usepackage{amsmath}
\usepackage{amssymb}
\usepackage{mathtools}
\usepackage{hyperref}

\usepackage[capitalise,noabbrev]{cleveref} 
\usepackage{geometry}
\usepackage{xcolor}
\usepackage{comment} 
\usepackage{todonotes}
\usepackage{graphicx}

\usepackage{lineno}

\title{The plus construction with respect to subrings of the rationals}
\author[G. Carrión Santiago]{Guille {Carri\'on~Santiago}}
\address{Departamento de Matemática Aplicada, Ciencia e Ingeniería de los Materiales y Tecnología Electrónica, Universidad Rey Juan Carlos, 28933 Móstoles, Madrid(Spain)}
\email{guille.carrions@urjc.es}
\author[R. Flores]{Ram\'on Flores}
\address{Departamento de Geometría y Topología, Facultad de Matemáticas, Universidad de \linebreak Sevilla, 41013 Sevilla (Spain)}
\email{ramonjflores@us.es}
\author[J. Scherer]{J\'er\^ome Scherer}
\address{Mathematics, Ecole Polytechnique F\'ed\'erale de Lausanne, EPFL, Switzerland}
\email{jerome.scherer@epfl.ch}



\newtheorem{theorem}{Theorem}[section]
\newtheorem{lemma}[theorem]{Lemma}
\newtheorem{prop}[theorem]{Proposition}

\theoremstyle{definition}
\newtheorem{definition}[theorem]{Definition}

\newtheorem{example}[theorem]{Example}
\newtheorem{remark}[theorem]{Remark}

\usepackage[
backend=biber,
style=alphabetic,
url=false,
isbn=false,
]{biblatex}
\newcommand{\Z}{{\mathbb Z}}

\newcommand{\C}{{\mathcal C}}

\newcommand{\Q}{{\mathbb Q}}
\newcommand{\N}{{\mathbb N}}

\newcommand{\Hom}{{\rm Hom}}

\newcommand{\cell}{{\rm cell}}

\begin{document}
\begin{abstract}
Let $J$ be any subset of the prime numbers. We construct an explicit model of a universal $H\Z[J^{-1}]$-acyclic group. Its presentation complex $\mathcal{M}$ is then a universal $H\Z[J^{-1}]$-acyclic space so that the corresponding nullification functor provides a functorial plus constructions for ordinary homology with $\Z[J^{-1}]$ coefficients. Motivated by classical results about Quillen's plus construction for integral homology, we prove that the $H \Z[J^{-1}]$-acyclization functor and the $\mathcal{M}$-cellularization functor coincide. We also show that the acyclization-plus construction fiber sequence is always a cofiber sequence for simply connected spaces, but almost never so when the plus construction is not simply connected, unlike in the classical case.
\end{abstract}
\maketitle

\section*{Introduction}
As early as in 1969, Kervaire \cite{Kervaire69} associated a smooth homology sphere  with a homotopy sphere. However, the plus construction as such is forever associated with the name of Quillen, as it was the fundamental tool used in the definition of higher algebraic K-theory groups, \cite{MR338129}. The plus construction, in its classical form, modifies a space by attaching $2$- and $3$-cells to kill the maximal perfect group while preserving its integral homology. The $2$-cells remove “non-abelian noise” from the fundamental group and the $3$-cells  ensure that the homological information remains unchanged.\medskip 

Despite the impact of Quillen's work, which quickly underscored the importance of this construction, it soon became evident that the original formulation had two significant limitations: on the one hand, the process of attaching cells to eliminate the perfect subgroup without altering the homology was not functorial; on the other, the construction was limited to ordinary homology with integer coefficients. Both problems were addressed by Bousfield through the development of homological localizations, more precisely via an $f$-localization, see \cite{Bousfield75} and also Farjoun's book \cite{Farjoun1996}.
Given a map $f\colon A \to B$ there is a universal construction, called homotopical localization and written $L_f$, that inverts $f$ up to homotopy, see \cref{Def:f-localization}.
\medskip

In the case of integral homology, it suffices to consider a wedge $E$ of representatives of all homotopy types of countable acyclic CW-complexes, and choose $f\colon E \to \ast$, see \cite[1.E.5]{Farjoun1996}, and \cite[Remark~5.1.1]{MR1318881}. The localization functor $L_f$ obtained in this way is called the $E$-nullification functor, denoted by $P_E$; see \cref{def:nullification}. For the chosen space $E$ it agrees up to homotopy with Quillen's plus construction. This means that the plus construction of a space $X$ can be done by repeatedly killing all maps from $E$ to $X$. Given a space $X$, there is a natural map $X\to P_E X$ that can be thought of as the largest (homotopy) quotient of $X$ which is totally invisible to~$E$.

As already noted in the above cited references, this method is particularly relevant because it allows one for the immediate generalization of the plus construction to any other arbitrary homology theory, as follows: given a homology theory $\mathbf h$, let $E_{\mathbf{h}}$ be a large enough $\mathbf h$-acyclic space (it should kill all $\mathbf h$-acyclic spaces in the sense of \cref{rem:naive}). It is known that such an $E_{\mathbf{h}}$ always exists, and the plus construction with respect to $\mathbf{h}$ is defined as the $E_{\mathbf{h}}$-nullification. The space $E_{\mathbf{h}}$ is referred to as a \emph{generator} of the plus construction.\medskip

The interpretation of the plus construction as a nullification was studied by Tai in \cite{MR1640095}. In the following years, significant efforts were devoted to understanding plus constructions associated with specific homology theories. For instance, Berrick and Casacuberta proved in \cite{Berrick1999} that the plus construction for integral homology is the nullification with respect to the classifying space of a certain explicit universal acyclic group $\mathcal F$, which serves as a generator for the integral case. Building on these ideas, Casacuberta, Rodríguez, and Scevenels \cite{Casacuberta1999} constructed a generator for ordinary mod-$p$ homology.
Later, Mislin and Peschke \cite{Mislin2001} extensively studied the plus constructions for an arbitrary homology theory $\mathbf h$. This perspective highlights the advantages of nullification functors over purely cellular constructions, it provides more direct access to low homotopy groups and enjoys better properties for studying the preservation of fibrations; see \cite{Farjoun1996}, \cite{Bousfield94}, \cite{Bousfield1997}. However, a fundamental question arises from the fact that the generators of the plus construction are not unique up to homotopy, making it both interesting and practical to find explicit models rather than relying solely on universal properties. While integral and mod-$p$ cases admit such descriptions, the group variety-based methods used so far do not adapt well to the rational case, for which no concrete model was currently known.\medskip

The main contribution of our study is the explicit description of a generator for the plus construction with respect to ordinary homology with coefficients in any subring of $\mathbb Q$. We explain here our approach for the rational case, but note, see \cref{rem:naive}, that a first more naive guess is discarded by the Promislow group, the same group Gardam used as a counterexample to the unit conjecture for group rings, \cite{MR4334981}.
\medskip

The construction is carried out in two steps. First, in a group theoretical context, we construct a group $\Gamma$ whose generators are, without going into full technical details, torsion elements up to commutators. That is, we impose relations of the form:
\[
x^{r}[x_1, x_2] \dots [x_{2n-1},x_{2n}]= 1
\]
where $r\in \N$ and $x_1,\dots, x_{2n}$ are other generators that in turn must be also torsion up to commutators, see Definitions~\ref{def:Grn} and \ref{def:Gamma}.
This group serves as a universal object 
for the group-theoretic version of the plus construction by mimicking the topological nullification: given a group $G$, the socle $S_\Gamma G$ is the normal subgroup of $G$ generated by all images of homomorphisms from $\Gamma$, and by repeating this process inductively one kills the $\Gamma$-radical $T_\Gamma G$ of $G$. The $\Gamma$-nullification functor is then defined as $P_\Gamma G = G/T_\Gamma G$.
\bigskip

\noindent
{\bf \cref{prop:group nullification}.}
\emph{
Let $G$ be any group. Then $T_\Gamma G$ coincides with the maximal
normal $H\Q$-perfect subgroup. Moreover, $S_\Gamma G = T_\Gamma G$ so that $P_\Gamma G$ is isomorphic to the quotient $G/S_\Gamma G$.}
\medskip

Second, now in the category of spaces, we verify that the presentation complex $\mathcal M$ of the previously constructed group $\Gamma$ is the desired generator, see \cref{def:universal Moore space}. This leads to our main result.
\medskip

\noindent
{\bf \cref{thm:plus}.}
\emph{The nullification functor $P_{\mathcal M}$ coincides with the $H\Q$-plus construction, that is, for every connected space $X$, the map $q\colon X\to P_{\mathcal M}X$ induces an isomorphism in homology with coefficients in $\Q$, and $\pi_1(q)$ is an epimorphism with kernel the $\Q$-perfect radical.
}\medskip

Moreover, we demonstrate that the space $\mathcal M$ is a Moore space: 
its only reduced homology group is $H_1(\mathcal M;\mathbb Z)\cong \bigoplus_{p, k} \Z/p^k$. Hence, its suspension is a wedge of copies of $M(\Z/p^k,2)$ where $p$ is a prime number and $k\geq 1$. Therefore, we obtain a generator whose homotopical structure is clear and concise.\medskip

The explicit construction of the generator $\mathcal M$ is useful to understand better the class of $H\Q$-acyclic spaces by means of the nullification functor $P_{\mathcal M}$. Clearly, by design, any $H\Q$-acyclic space $X$ is $\mathcal M$-acyclic, i.e., the $\mathcal M$-nullification $P_{\mathcal M} X$ is contractible. There is another class, which is in principle smaller, namely that of $\mathcal M$-cellular spaces, which can be built from $\mathcal M$ by using only wedges, pushouts, and telescopes, in a analogous way to how CW-complexes are constructed from spheres, see \cref{def:cell-acyclic}. 
The best $\mathcal M$-cellular approximation can be carried out functorially via the cellularization functor $\cell_{\mathcal M}$, which serves as a counterpart to the nullification functor: if the map $X \to P_{\mathcal M} X$ is the ``largest quotient'' of $X$ that is invisible to $\mathcal M$, then the map $\cell_{\mathcal M} X \to X$ identifies as the ``largest subobject'' that is entirely visible to $\mathcal M$. \medskip

A delicate problem in general is to understand the difference between the class of $E$-cellular spaces for a given space $E$ and its corresponding $E$-acyclic class (see \cref{def:classes}).
In this context, Chachólski, Parent, and Stanley \cite{Cha04}, as well as the second author and Rodríguez \cite{Flores21}, provide a cellular generator for any acyclic class in the categories of spaces and groups, respectively. It is worth noting that the cellular and acyclic classes generated by torsion Moore spaces differ in general, \cite[Example~4.9]{MR1408539}. In this article, we prove that they agree for the Moore space $\mathcal M$.
\medskip

\noindent
{\bf \cref{thm:acyclic=cellular}.}
\emph{The three following classes coincide: $\mathcal M$-cellular spaces, $\mathcal M$-acyclic spaces, 
and $H\mathbb Q$-acyclic spaces. In particular, the functor $cell_{\mathcal M}$ is $H\mathbb Q$-acyclization.
}\medskip

We conclude this article by highlighting a characteristic feature of Quillen's plus construction that is not shared by the rational one. Recall that the classical acyclization fiber sequence 
\[
AX \rightarrow X \rightarrow X^+
\] 
is always a cofibration, a phenomenon studied in a wider context by Alonso in \cite{Alonso1,Alonso2}, see also \cite{MR1358818}. Even if this is not the case for the rational acyclization fiber sequence $A_{H\Q}X\rightarrow X\rightarrow X^{+H\Q} = P_{\mathcal M} X$ in general, as illustrated in \cref{ex:not a cofibration}, we understand completely the situation as explained in our last result.
\medskip

\noindent
{\bf \cref{thm:cofibration}.}
\emph{
Let $X$ be a connected space. Then the acyclization fiber sequence 
\[
A_{H\mathbb Q} X \rightarrow X \rightarrow X^{+ H\mathbb Q}
\]
is a cofiber sequence if and only if one of the two following conditions holds:
\begin{itemize}
    \item[(a)] The $H \mathbb Q$-plus construction $X^{+  H\mathbb Q}$ is simply connected;
    \item[(b)] the $H\mathbb Q$-acyclization $A_{H\mathbb Q} X$ is $H \mathbb Z$-acyclic. 
\end{itemize}
}\medskip

While most of the literature on plus constructions dates back to the 1990s and early 2000s, recent developments have introduced generalizations of the plus construction \emph{à la Quillen}, involving the attachment of low-dimensional cells for rings not necessarily equal to $\Z$. Ye views plus constructions, in their cell-by-cell form, as one of various ways to realize certain prescribed homomorphisms at the level of fundamental groups while leaving higher homology groups unchanged \cite[Theorem~1.1]{MR3009739}. Later, using a similar approach, Broto, Levi, and Oliver describe a relative plus construction \cite{Broto2021}. In contrast, our approach contributes to the understanding of the plus construction from the perspective of Bousfield localization. Rather than following the cell-attachment paradigm, we interpret the plus construction as a nullification process. This paper is part of a larger project aiming at understanding plus constructions; some of the questions we answer here were originated in \cite{MR4602845}.
\medskip

\noindent
{\bf Acknowledgments.}
We thank Ian Leary for pointing out the relationship of \cref{rem:naive} with the unit conjecture for group rings. We also thank the referee for helpful remarks.

The first author was supported by Universidad de Málaga grant G RYC-2010-05663, Comissionat per Universitats i Recerca de la Generalitat de Catalunya (grant No. 2021-SGR-01015), and MICINN grant PID2020-116481GB-I00, and PID2023-149804NB-I00. 

The second author wishes to thank the Laboratory for Topology and Neuroscience at EPFL for its hospitality during the preparation of this paper, and was partially supported by the MICINN grant PID2020-117971GB-C21.

The third author would also like to thank the Isaac Newton Institute for Mathematical Sciences, Cambridge, for support and hospitality during the program Equivariant Homotopy Theory in Context, where this paper was finished. This work was supported by EPSRC grant no EP/R014604/1.

\section{Preliminaries}
\label{Section:Preliminaries}

In this section, we provide the background needed for the rest of this note. 
Our convention is that ``space" means ``simplicial set" and that all maps are simplicial.

\subsection{Localization, nullification, and plus constructions}
As we deal mainly with localizations and related functors, we start with the definition of $f$-localization, as proposed by Farjoun in \cite[Chapter 1]{Farjoun1996}:

\begin{definition}
\label{Def:f-localization}
Let $X$ be a pointed space and $f\colon A\rightarrow B$ a map. We say that $X$ is $f$-\emph{local} if the map induced by composition $map_*(B,X)\rightarrow map_*(A,X)$ is a weak homotopy equivalence. For every space $X$, the \emph{$f$-localization} of $X$ is an $f$-local space $L_fX$ together with a localization map $\eta\colon X\rightarrow L_fX$ with the property that every map $h\colon X\rightarrow Y$ to an $f$-local space factors up to homotopy through a unique map $g\colon L_fX\rightarrow Y$ as $h \simeq g\circ \eta$.   
\end{definition}

There is a well-defined coaugmented homotopy idempotent endofunctor $L_f$ in the category of spaces that sends every space to its $f$-localization.
For the construction of the $f$-localization and its main properties, the reader is referred to the first chapter of \cite{Farjoun1996}.

Two examples of $f$-localizations are very important to us, namely nullifications and homological localizations.

\begin{definition}
\label{def:nullification}
When $f\colon A\rightarrow \ast$ is the constant map, the $f$-localization will be denoted by $P_A$ and called $A$-\emph{nullification}. 
\end{definition}

The main properties of $A$-nullification were studied
by Bousfield, who called this functor $A$-periodization, because of its importance in $v_1$-periodic homotopy theory, see \cite{Bousfield94}.

\begin{definition}
Given a generalized homology theory $\mathbf{h}$, the \emph{homological localization} with respect to $\mathbf{h}$, denoted by $L_\mathbf{h}$, is the coaugmented idempotent functor in the category of spaces which is characterized by the fact that for every space $X$, the coaugmentation $X\rightarrow L_\mathbf{h}X$ is terminal among all $\mathbf{h}$-homology equivalences $X \to Y$.
\end{definition}

Such a functor was constructed by Bousfield \cite{Bousfield75} as an $f$-localization with respect to a so-called \emph{universal $\mathbf{h}$-equivalence} $f$ (see also \cite[Section 6]{MR1290581}). This map can be chosen to be a wedge of set representatives of homotopy classes of $\mathbf{h}$-equivalences $A\rightarrow B$, where the number of simplices of $A$ and $B$ is not bigger than the cardinality of $\mathbf{h}(\ast)$.

The main objects of study of our paper are the nullification functors associated to homology localization functors.
\begin{definition}
\label{def:plusconstruction}
Given a generalized homology theory $\mathbf{h}$, consider a universal $\mathbf{h}$-equivalence $f$, and its homotopy cofiber $E_{\mathbf{h}}$. Then the nullification functor $P_{E_{\mathbf{h}}}$ is called the \emph{plus construction} with respect to $\mathbf{h}$. The value $P_{E_{\mathbf{h}}}(X)$ is usually written as $X^{+\mathbf{h}}$.
\end{definition}

This definition generalizes the classical Quillen plus construction, which corresponds to the case of ordinary homology with coefficients in the integers. 
In this paper, we will focus on ordinary homology with coefficients in subrings of the rationals. It is worth noting here that the map $X\rightarrow X^{+\mathbf{h}}$ is an $\mathbf{h}$-equivalence, and hence the homological localization factors through it by universality.

Moreover, a generalization of perfectness provides information about the fundamental group of the plus-construction.
The following definitions and results correspond to Mislin and Peschke's more general ones in \cite{Mislin2001}. They study a notion of $\mathbf{h}n$-perfectness while we only focus on the case $n=1$ and thus drop the $1$ from the notation.

For a generalized homology theory $\mathbf{h}$ and a group $G$, denote by $P^{\mathbf{h}}G$ the kernel of the homomorphism $\pi_1 K(G,1)\rightarrow \pi_1K(G,1)^{+\mathbf{h}}$ induced by plus construction. Then we have the following:
\begin{definition}
\label{def:hperfect}
A group $G$ is called $\mathbf{h}$-\emph{perfect} if $P^{\mathbf{h}}G=G$.
\end{definition}

\begin{remark}
\label{rmk:perfect}
    Notice that if $\mathbf{h}$ is an ordinary homology theory with coefficients in a ring $R$ such that $R$ is either of non-zero characteristic, or a subring of the rationals $\Q$, then $P^\mathbf{h}G$ is the maximal $R$-perfect subgroup of $G$. In fact, $\mathbf{h}$-perfect means $R$-perfect in this case, see \cite[Section 3.3]{MR4602845}. In general, for an arbitrary homology theory $\mathbf{h}$, $P^{\mathbf{h}}G$ is the maximal $\mathbf{h}$-perfect subgroup of $G$. 
\end{remark}

For the plus-construction of a general space $X$, Mislin and Peschke proved the following:
\begin{prop}
{\rm \cite[Proposition~2.3 (iii)]{Mislin2001}}\label{prop:hperfect}
For any space $X$, the kernel of the homomorphism ${\pi_1X\rightarrow \pi_1X^{+\mathbf{h}}}$ is \mbox{$\mathbf{h}$-perfect.}
\end{prop}

In particular, if $\mathbf{h}$ is a connective homology theory, it is known that ${\pi_1X^\mathbf{h}=\pi_1X/P^\mathbf{h}\pi_1X}$. This is the case of rational homology, which is the main object of interest for us here. For arbitrary homology theories, it is not known if the latter is true; see \cite[Section 2]{Mislin2001} for details.

Let us finally mention that, in general, the plus construction does not coincide with homological localization. For rational homology, it is clear that $P^{H\Q}(\pi_1(S^1))=1$, hence $(S^1)^{+H\Q}\cong S^1$. However, $L_{H\Q}S^1$ is $K(\Q,1)$. For any subring of the rationals, the following more sophisticated example shows that the difference between the plus construction and homological localization can be important:

\begin{example}
Let $F$ be a free group of rank at least $2$ and $R$ be any subring of $\Q$.
Notice that $F$ does not contain any $HR$-perfect subgroup since it is free (see \cref{rmk:perfect}), so on the one hand, $K(F,1)^{+HR}$ is exactly $K(F,1)$. On the other hand, Ivanov and Mikhailov \cite{MR4601076} show that the $HR$-length of $F$ is at least $\omega+\omega$. 
The homological localization $L^R(F)$ is obtained as a limit of a tower whose stages are the quotients by the transfinite lower $R$-central series, as introduced by Bousfield in \cite{MR447375}. The length of this tower is thus at least $\omega + \omega$ before it stabilizes. In particular, $F$ is not $HR$-local and this implies that $\pi_1 (L_{HR}(K(F,1)))\cong L^R F\not\cong F$. 
\end{example}

\subsection{Cellular and acyclic classes}
\label{subsec:classes}
In the 90's, Chach\'olski \cite{MR1408539} and Farjoun \cite{Farjoun1996} developed the concept of closed classes; two of them are particularly relevant in the study of homotopical localization.

\begin{definition}
\label{def:classes}
A class $\C$ of pointed spaces is called \emph{cellular} if it is closed under weak equivalences and taking pointed homotopy colimits. It is \emph{acyclic} if it is moreover closed under extensions by fibrations. The cellular class generated by a pointed space $E$ is denoted by $\C(E)$ and the corresponding acyclic class is denoted by $\overline{\C(E)}$.
\end{definition}

\begin{definition}
\label{def:cell-acyclic}
Let $E$ and $X$ be pointed spaces. We will say that $X$ is $E$-\emph{cellular}, or that $E$ \emph{builds} $X$, if $X$ belongs to $\C(E)$. Likewise, we say that $X$ is $E$-\emph{acyclic}, or that $E$ \emph{kills} $X$, if $X$ belongs to $\overline{\C(E)}$.
\end{definition}

Given a pointed space $E$, Chach\'olski and Farjoun define a cellularization functor $cell_E$ in the category of spaces. It is idempotent, augmented, and characterized by the following facts: 
\begin{enumerate}
\item[(a)] for every $X$, the space $cell_EX$ is $E$-cellular;
\item[(b)] the augmentation induces an equivalence 
$map_*(E,cell_E X)\simeq map_*(E,X).$
\end{enumerate}
In particular, a space $X$ belongs to $\C(E)$ if and only if $cell_EX\simeq X$. The values of $cell_E$ can be constructed in a very explicit way:

\begin{theorem}[{\cite[Theorem 20.3]{MR1408539}}]
\label{thm:Chacholski}
Let $X \rightarrow X'$ be a map of connected and pointed spaces with $E$-cellular homotopy fiber. Assume that it induces the trivial map on $[E, -]_*$. Then $cell_E X$ coincides with the homotopy fiber of the composite map $X \rightarrow X' \rightarrow P_{\Sigma E} X'$.
\end{theorem}

This result provides a recognition principle for spaces in the cellular class generated by $E$, and applies in particular to the case where $X'$ is the homotopy cofiber of the evaluation map $\vee E \rightarrow X$, see \cite[Theorem 20.5]{MR1408539}. Here, the wedge is taken over representatives of all homotopy classes of pointed maps $E \rightarrow X$.

In a somewhat dual way, it can be seen that a pointed space $X$ is $E$-acyclic if and only if $P_EX$ is contractible. Hence, if we consider the augmented functor $\overline{P}_E$ that takes every space to the fiber of the $E$-nullification, a space $X$ is $E$-acyclic if and only if $\overline{P}_E X\simeq X$. This implies in particular that the cellular and acyclic classes can be studied through the co-reflections $cell_E$ and $\overline{P}_E$.

\subsection{Acyclicity and plus constructions}
Next, we will review some relations between acyclicity in the previous sense and in the homological sense. Recall that given a generalized homology theory $\mathbf{h}$, a space $X$ is $\mathbf{h}$-acyclic if its homology is trivial. 
\begin{definition}
Let $\mathbf{h}$ be a homology theory and $E$ be an $\mathbf{h}$-acyclic space. We say that $E$ is \emph{universal} (with respect to $\mathbf{h}$) if $E$ kills every $\mathbf{h}$-acyclic space.
\end{definition}

It is a direct consequence of \cite[1.E.5]{Farjoun1996} that the wedge of representatives of homotopy types of (small) $\mathbf{h}$-acyclic spaces as the space $E_{\mathbf{h}}$ in \cref{def:plusconstruction} is universal and that any universal space can be used to obtain the plus construction:
\begin{prop}
Let $\mathbf{h}$ be a homology theory and $E$ be a universal $\mathbf{h}$-acyclic space. Then, for every space $X$, the $E$-nullification 
$X\to P_E X$, coincides with the $\mathbf{h}$-plus construction.
\end{prop}
If $E$ is a universal $\mathbf{h}$-acyclic space we use the notation  $A_{\mathbf{h}}$ for the acyclization functor $\overline P_E$, that is, the homotopy fiber of the map $X\to X^{+\mathbf{h}}$.
Now, since the plus construction $q\colon X\to X^{+\mathbf{h}}=P_E X$ is a nullification, it is initial with respect to all morphisms $X\to Y$ with $Y$ an $E$-null space. Here, for practical reasons, we are interested in another universal property related to acyclicity. 

\begin{definition}
Let $\mathbf{h}$ be a homology theory. A map of connected spaces $f\colon X\to Y$ is $\mathbf{h}$-acyclic if its homotopy fiber $\mathrm{Fib}(f)$ is $\mathbf{h}$-acyclic.
\end{definition}
\begin{prop}
Let $\mathbf{h}$ be any homology theory. Then, the plus construction ${q\colon X\to X^{+h}}$ is terminal with respect to all $h$-acyclic maps $X\to Y$.
\end{prop}
\begin{proof}
Consider any $\mathbf{h}$-acyclic map $f\colon X \rightarrow Y$ and apply fiberwise nullification to it, see \cite[Theorem~1.F.1]{Farjoun1996}. Since $P_E$ kills the homotopy fiber, we see that $f$ is a $P_E$-equivalence. This implies that $X \rightarrow P_E X \simeq P_E Y$ factors through $f$.
\end{proof}

\subsection{Nullification in the category of groups}

Since the first step of our approach to constructing explicit models for the rational plus construction is about finding a universal group, it will be handy to use the group theoretical notions analogous to those presented above for spaces. This is inspired by Berrick and Casacuberta's approach in \cite{Berrick1999}. A nice survey about localization of groups can be found in \cite{MR1796125}.

Let us start with the $E$-nullification functor, where $E$ is any group. We wish to kill all maps from $E$ and repeat the process until the resulting quotient is $E$-local.
\begin{definition}
\label{def:socle}
Given groups $E$ and $G$, the $E$-\emph{socle} of $G$ is the normal subgroup $S_E(G)$ of $G$ generated by the images of all homomorphisms $E\rightarrow G$.
\end{definition}

By iterating, possibly transfinitely, this construction, we obtain quotient maps
\[
G \rightarrow G/S_E(G) \rightarrow (G/S_E(G))/S_E(G/S_E(G)) \rightarrow \dots \rightarrow G/T_E(G),
\]
where $T_E(G)$ is defined in the following way and we refer the reader to \cite[Theorem 3.2]{MR1320986} where more details 
can be found.

\begin{definition}
\label{def:radical}
The $E$-\emph{radical} $T_EG$ of $G$ is the normal subgroup of $G$ such that $G/T_EG$ is the largest quotient of $G$ with $\Hom(E, G/T_EG) = 0$. This quotient $G/T_E G =: P_E G$ is the $E$-\emph{nullification} of $G$.
\end{definition}

The concept of closed class also has its immediate counterpart in groups:
The cellular class $\mathcal{C}(E)$ is the smallest class of groups containing $E$, and closed under isomorphisms and colimits. The acyclic class $\overline{\mathcal{C}(E)}$, consisting of the $E$-radical groups, is the smallest one that is moreover closed under group extensions.

\section{The universal group}
\label{Section:The group}
Since the $H\Z[J^{-1}]$-plus construction kills all $H\Z$-acyclic spaces, and also all $p$-torsion spaces for $p\in J$, a first naive guess would be to approximate the rational plus construction by taking nullification with respect to $B\mathcal F$, Berrick and Casacuberta's universal $H\Z$-acyclic space, and add all torsion Moore spaces $M(C_p, 1)$, for all primes $p\in J$. This does not quite do the job we want, for, at the group theoretical level already, we cannot kill all torsion elements up to commutators by using homomorphisms out of $\mathcal F$ and cyclic groups $C_p$, as shown in the following example when $J$ equals the set of all prime numbers.

\begin{example}
\label{rem:naive}
Let $G$ be the \emph{Promislow group}, i.e. the Bieberbach group of
Hirsch rank three appearing in Strojnowski's \cite[Theorem~2.6]{MR1477181}. This group $G$ is defined by a presentation 
\[
\langle x, y \, \mid \, x^{-1} y^2 x = y^{-2}, y^{-1} x^2 y = x^{-2} \rangle,
\]
and serves as a counterexample to the unit conjecture for group rings in Gardam's article \cite{MR4334981}. The subtle properties used by Gardam, related to having unique products or being diffuse, do not seem to be apparently related to the absence of torsion and the structure of the commutator subgroup, which matter to us.

On the one hand the group $G$ is torsion-free; hence, there are no maps from $C_p$ to $G$ for any prime $p$.
On the other hand the subgroup usually called $\Delta$ in the theory of Bieberbach groups is free abelian of rank three, generated by $x^2, y^2, (xy)^2$. The commutator subgroup $[G,G]$,
and then also any perfect subgroup, must therefore be contained in this free abelian group, implying that the maximal perfect subgroup $\mathcal P G$ is trivial. In particular, every homomorphisms from the Berrick-Casacuberta group $\mathcal{F}$ to $G$ is trivial.

The previous items imply that $G$ is local with respect to the free product of $\mathcal F$ and all cyclic groups $C_p$ of prime order~$p$. However the generators $x$ and $y$ are torsion up to commutators; in fact, the abelianization of $G$ is isomorphic to $\mathbb Z/4 \times \mathbb Z/4$.
Hence, $G$ is $H \Q$-perfect.
The same happens with the combinatorial Hantzsche-Wendt groups $G_n$, for $n \geq 2$, introduced by Craig and Linnell in \cite{MR4381283} and studied in \cite{MR4379387} by Popko and Szczepa\'nski. The group $G_2$ is the Promislow group $G$ above and $G_n$ has $n$ generators in general, subject to analogous ``torsion up to commutator'' relations as~$G$.
\end{example}

As hinted at in this example the construction of a universal $H\Z[J^{-1}]$-acyclic group must take into account torsion elements up to commutators. Let $J$ be a set of prime numbers and $\Z[J^{-1}]$ be the corresponding subring of the rational numbers $\Q$. Recall from \cref{def:hperfect} and \cref{prop:hperfect} that  a group $G$ is called $H\Z[J^{-1}]$-perfect if, 
$H_1(G;\Z[J^{-1}])\cong G^{\operatorname{ab}}\otimes \Z[J^{-1}]=0$.
In this section, we construct a group $\Gamma$ such that the nullification $P_\Gamma$ realizes the $\Z[J^{-1}]$-perfect reduction, i.e., $\overline{{\mathcal C}(\Gamma)}$ is the class of all $\Z[J^{-1}]$-perfect groups. Moreover, we will see that this nullification can be made in one step, i.e., the $\Gamma$-socle coincides with the $\Gamma$-radical as defined in \cref{def:socle} and \cref{def:radical}.

We write $\N(J)$ for the subset of $J$-torsion integers in $\N$, i.e., those integers that decompose as products of powers of primes in $J$. Notice that any element of a $\Z[J^{-1}] $-perfect group $G$  is $J$-torsion up to an element of the commutator subgroup of $G$. Let us thus construct such groups that are, loosely speaking, freely generated by one generator verifying the above property and such that all subsequent elements appearing in the commutator relations also verify it. 

More concretely, we start with a generator $x(0)$, which we imagine in the zeroth layer of our construction. If the group we construct is $\Z[J^{-1}]$-perfect, some $J$-power of $x(0)$ must live in the commutator subgroup.
We choose thus a number $r(0) \in \N (J)$ such that $x(0)^{r(0)}$ is a product of commutators. There exists thus
an integer $n(0)$, the number of commutators we need (possibly zero), and elements $x(1, 1), \dots, x(1, 2n(0))$, which we imagine in the first layer of the construction, such that the following relation holds:
\[
x(0)^{r(0)}[x(1, 1), x(1, 2)] \dots [x(1, 2n(0)-1), x(1, 2n(0))] = 1
\label{eq:R0} \tag{$R(0)$}
\]
If present, the new generators $x(1, i)$ must also be of $J$-torsion up to commutators.
Choose $r(1, i)  \in \N(J)$ and $n(1, i) \in \N$ for $1 \leq i \leq 2n(0)$ and introduce new generators $x(2, i, 1),$ $\dots,$ $x(2, i, 2n(1, i))$ such that the following relations hold for \mbox{all~$i$:}
\[
x(1, i)^{r(1, i)}[x(2, i, 1), x(2, i, 2)] \dots [x(2, i, 2n(1, i)-1), x(2, i, 2n(1, i))] = 1
\label{eq:R1} \tag{$R(1, i)$}
\]
We iterate this procedure, which can stop at some point if all generators from the last constructed layer are torsion
(no commutators are needed). In general, we introduce integers
$r(k, i_1, \dots, i_k) \in \N (J)$ for $k\geq 1$, integers $n(k, i_1, \dots, i_k)$ and generators in the $(k+1)$-st layer
$x(k+1, i_1, \dots, i_k, i_{k+1})$ such that:
\[
x(k, i_1, \dots, i_k)^{r(k, i_1, \dots, i_k)}\prod_{i=1}^{n(k, i_1, \dots, i_k)}
[x(k+1, i_1, \dots, i_k, 2i-1), x(k+1, i_1, \dots, i_k, 2i)]=1,
\]
we call this relation $R(k, i_1, \dots, i_k)$. 

\begin{definition}
\label{def:Grn}
For all choices of $J$-torsion integers 
\[
\underline{r} = (r(0), r(1, 1), \dots, r(1, 2n(0)), r(2, 1, 1), \dots )
\]
and integers $\underline{n} = (n(0), n(1, 1), \dots, n(1, n(0)), \dots)$ we define a group by the presentation
\[
G(\underline{r}, \underline{n}) = \langle x(0), x(\underline{n}) \, \mid \, R(0), R(\underline{r}) \rangle
\]
where $x(\underline{n})$ denotes a sequence of generators, as many as 
elements in the sequence $\underline{n}$, and likewise $R(\underline{r})$
denotes a family of relations, as many as elements in $\underline{r}$.
\end{definition}

Let us provide a few examples to make these groups more concrete.

\begin{example}
\label{ex:cyclic}
When $r(0) = p^k$ is a prime power in $J$ and $n(0) = 0$, the sequence stops
and we get the presentation $\langle x(0) \, \mid \, x(0)^{p^k} \rangle$
of the cyclic group $C_{p^k}$ of order~$p^k$.
\end{example}

\begin{example}
\label{ex:BC}
When $\underline{r}$ is the constant sequence made of $1$'s and 
$\underline{n}$ is any sequence indicating that every generator is
a product of a certain number of commutators, we get the presentation of
an $H\Z$-acyclic group as in Berrick and Casacuberta's beautiful 
article \cite[]{Berrick1999}.
\end{example}

By design, the groups $G(\underline{r}, \underline{n})$ have the property 
that all generators are torsion up to an element of the commutator 
subgroup. This implies directly the next lemma.
\begin{lemma}
\label{lem:HJacyclic}
Let $J$ be a set of primes and $\underline{r}$ and $\underline{n}$ be a pair of sequences as in \cref{def:Grn}. Then, the group $G(\underline{r}, \underline{n})$ is $H \Z [ J^{-1} ]$-perfect.
\hfill{\qed}
\end{lemma}

The universal group we introduce next depends on the set $J$, but we drop $J$ from the notation since it will always be clear from the context.

\begin{definition}
\label{def:Gamma}
Let $J$ be a set of primes. We define $\Gamma = \Gamma_J$ as the free product of all groups $G(\underline{r}, \underline{n})$ over every possible choice of sequences $\underline{r}$ and $ \underline{n}$ as in \cref{def:Grn}. 
\end{definition}

It follows that the group $\Gamma$ is a generator for the $H\mathbb Z[J^{-1}]$-perfect reduction. The associated nullification functor kills the desired subgroup of the fundamental group in a single step, as we show now.

\begin{theorem}
\label{prop:group nullification}
Let $G$ be any group. Then $T_\Gamma G$ coincides with the maximal
normal $H\Q$-perfect subgroup. Moreover, $S_\Gamma G = T_\Gamma G$ so that $P_\Gamma G$ is isomorphic to the quotient $G/S_\Gamma G$.
\end{theorem}

\begin{proof}
Any $H\Z[J^{-1}]$-perfect subgroup of $G$ consists of $J$-torsion elements up to commutators. These elements are, therefore, all hit by some homomorphism from a $G(\underline{r}, \underline{n})$. Conversely, any quotient of $\Gamma$ shares this property. This proves both claims.
Notice that any conjugate of a torsion element up to commutators is also torsion up to commutators, so the maximal $H\Z[J^{-1}]$-perfect subgroup must be normal.
\end{proof}

\begin{example}
Going back to $G = \langle x, y \, \mid \, x^{-1} y^2 x = y^{-2}, y^{-1} x^2 y = x^{-2} \rangle$, the Promislow group in \cref{rem:naive}, we can choose ${\underline{r} = (4, 2,4, 2, 4, \dots )}$ to be an alternating sequence and $\underline{n} = (1, 1, 1, \dots)$ to be constant. Then, the group $G(\underline{r}, \underline{n})$ admits two homomorphisms to $G$ hitting the generators $x$ and $y$ respectively. The first one $\varphi$ sends $x(0)$ to $x$ and, since $x^4 = [x^2, y^{-1}]$, we continue with $\varphi(x(1, 1)) = x^2$ and $\varphi(x(1, 2)) = y^{-1}$. The square of $x^2$ is again this same commutator and $y^{-4} = [x^{-1}, y^2]$, so we can iterate. The second one does the same starting from $\psi(x(0)) = y$. This implies that $T_\Gamma G=G$ and, therefore, $P_\Gamma G$ is trivial.
\end{example}

\section{A universal \texorpdfstring{$H \Z[J^{-1}]$}{HZ[J-1]}-acyclic space}
\label{Section:The space}
In this section, we show that the presentation complex of the group $\Gamma$ defined in \cref{def:Gamma} is a universal $H\Z[J^{-1}]$-acyclic space which has the homotopy type of a \emph{Moore space} of type $M(\Gamma, 1)$. This means that its fundamental group is isomorphic to $\Gamma$ and its only non-trivial reduced homology group is the first one.

\begin{definition}
\label{def:Moore space}
Let $J$ be a set of primes and  $\underline{r}$ and $\underline{n}$ be sequences as in \cref{def:Grn}. Then, the space $M(\underline{r}, \underline{n})$ is the $2$-dimensional \emph{presentation complex} obtained from a wedge $W$ of as many circles as there are generators $x(\underline{r})$ and by attaching one $2$-cell for each relator $R(\underline{r})$, the attaching map being a representative of the homotopy class $R(\underline{r})$ in the free group~$\pi_1 W$.
\end{definition}

Unlike the argument in \cite{Berrick1999}, we cannot conclude that these spaces are Moore spaces from the fact that their fundamental groups are locally free, as they are not so in general.

\begin{lemma}
\label{lem:Moore space}
Let $J$ be a set of primes and $\underline{r}$ and $\underline{n}$ be a pair of sequences as in \cref{def:Grn}. Then, the
space $M(\underline{r}, \underline{n})$ is an $H \Z[J^{-1}]$-acyclic
Moore space.
\end{lemma}
\begin{proof}
We only need to compute the second homology group of this $2$-dimensional CW-complex since we already know from \cref{lem:HJacyclic} that the fundamental group is 
$H \Z[J^{-1}]$-perfect. In fact, the first homology group is an abelian group generated by the $x(\underline{r})$'s with the relations implying that they are $J$-torsion.

To compute $H_2(M(\underline{r}, \underline{n}); \Z)$ we simply write
down the cellular chain complex. In degree two we have a free abelian group on as many generators as we have relators. The image
under the differential of the generator corresponding to the relator
$R(\underline{r})$ is $r(\underline{r})$ times the generator
corresponding to the generator $x(\underline{r})$. The differential is thus given
by a (possibly infinite-dimensional) matrix that is diagonal, in particular injective. 
This shows that $H_2(M(\underline{r}, \underline{n}); \Z) = 0$, so $M(\underline{r}, \underline{n})$ is a Moore space. A fortiori $H_2(M(\underline{r}, \underline{n}); \Z[J^{-1}]) = 0$ as well, hence it is $H \Z[J^{-1}]$-acyclic.
\end{proof}

\begin{definition}
\label{def:universal Moore space}
Let $J$ be a set of primes. We define $\mathcal M_J$ as the wedge of all Moore spaces $M(\underline{r}, \underline{n})$ over every possible choice of sequences $\underline{r}$ and $ \underline{n}$, see \cref{def:Moore space}. When the set of primes $J$ is clear from context, we omit it from the notation and write $\mathcal M_J$ simply as $\mathcal M$.
\end{definition}
\begin{remark}
\label{rem:Moore}
By construction, $\mathcal M$ is a Moore space with fundamental group $\pi_1 \mathcal M \cong \Gamma$ and $H_1(\mathcal M,\mathbb{Z})$
is a direct sum of multiple copies of cyclic groups $\mathbb{Z}/{p^k}$ where $p$ runs through all the prime numbers in $J$.
\end{remark}

From the previously obtained group theoretical properties we obtain now analogous homotopical ones for the $\mathcal M$-nullification.
First, when we compute $P_{\mathcal M} X$, we get the right fundamental group in one single step (one cofiber sequence), and next, iterating
the procedure of killing maps out of $\mathcal M$ and its suspensions, we
obtain the desired plus construction.

\begin{prop}
\label{prop:effect on pi1}
Let $J$ be a set of primes and $\mathcal M$ be the universal Moore space as in \cref{def:universal Moore space}. Then, for any connected space $X$, there is an isomorphism
\[\pi_1 P_{\mathcal M} X \cong P_\Gamma (\pi_1 X),\]
which is realized by the cofiber
sequence $\bigvee_{[M,X]_{*}} \mathcal M \rightarrow X \rightarrow X'$.
\end{prop}

\begin{proof}
Since $\mathcal M$ is a $2$-dimensional CW-complex, any homomorphism $\varphi\colon \Gamma \rightarrow \pi_1 X$ can be realized by a map $f\colon \mathcal M \rightarrow X$. 
Consider then the cofiber sequence
\[
\bigvee \mathcal M \rightarrow X \rightarrow X'\, ,
\]
where the wedge is taken over all homotopy classes of pointed maps $\mathcal M\to X$, and conclude by \cref{prop:group nullification} and the Seifert-van Kampen Theorem that $\pi_1 X' \cong P_\Gamma (\pi_1 X)$. 

Observe next that any
map $\mathcal M \rightarrow X'$ induces the trivial
map on the fundamental group since $\pi_1 X'$ is $\Gamma$-local. 
Therefore, $P_{\mathcal M} X = P_{\mathcal M} X'$ has the same fundamental group as $X'$.
\end{proof}
The previous result gives rise to a description of the $H\Z[J^{-1}]$-plus construction in terms of nullification using the space $\mathcal M$:
\begin{theorem}
\label{thm:plus}
Let $J$ be a set of primes and $\mathcal M$ be the universal Moore space as in \cref{def:universal Moore space}. The nullification functor $P_{\mathcal M}$ coincides with the $H\Z[J^{-1}]$-plus construction, that is, for every connected space $X$, the map $q\colon X\to P_{\mathcal M}X$ induces an isomorphism in homology with coefficients in $\Z[J^{-1}]$, and $\pi_1(q)$ is an epimorphism with kernel the $\Z[J^{-1}]$-perfect radical.
\end{theorem}
\begin{proof}
The homology isomorphism follows from the acyclicity of $\mathcal M$ and the effect on the fundamental group is described in \cref{prop:effect on pi1}.
\end{proof}

We conclude this section by highlighting a difference with the classical integral plus construction, which might be evident at this point. The cofiber sequence in \cref{prop:effect on pi1} gives the plus construction in one single step when $J$ is empty, but not so in general. Even though all maps from $\mathcal M$ to $X'$ induce the trivial map on fundamental groups, there might be non-trivial maps from higher suspensions $\Sigma^k \mathcal M \rightarrow X'$ one needs to kill to construct $X^{+H\mathbb Z[J^{-1}]}$. This would be the case for $X = M(\mathbb Z/p, n)$ a Moore space when $n \geq 2$ and $p \in J$.

\section{The \texorpdfstring{$H\mathbb Z[J^{-1}]$}{HZ[J-1]}-acyclization functor}
\label{Section:acyclization}

Our final objective is to identify
the $H\mathbb Z[J^{-1}]$-acyclization with the $\mathcal M$-cellularization. To get a better intuition of the situation, we start with a simple but enlightening case, that of the $2$-sphere, when $J$ consists of all prime numbers.

\begin{example}
\label{ex:sphere}
Let $\mathcal M$ be the space defined in \cref{def:universal Moore space}. When $J$ consists of all prime numbers, the \mbox{$\mathcal M$-nullification} functor is the rational plus construction, $P_{\mathcal M} X = X^{+ H\mathbb Q}$, by \cref{thm:plus}. Recall from Subsection~\ref{subsec:classes} that $\overline{P}_M(S^2)$ is the homotopy fiber of the \mbox{$M$-nullification of $S^2$.}

We use the fact that $M(\Q/\Z, 1)$ is $\mathcal M$-cellular, being a wedge of $M(\Z_{p^\infty}, 1)$'s, which in turn are filtered colimits of $M(\Z/p^n, 1)$'s. Consider the iterated cofiber Puppe sequence
\[
S^1 \rightarrow M(\Q, 1) \rightarrow M(\Q/\Z, 1) \rightarrow S^2 \rightarrow M(\Q, 2)
\]
which exhibits $M(\Q, 2)$ as a candidate for $(S^2)'$ in Chach\'olski's fiber sequence \cref{thm:Chacholski} 
describing the cellularization of $S^2$. In particular, the homotopy fiber of the map
\[
S^2 \rightarrow M(\Q, 2) = (S^2)_{(0)} 
\]
is $\mathcal M$-cellular by \cite[Corollary~9.A.10]{Farjoun1996}: More explicitly, the Moore space $M(\Q/\Z, 1)$ is $\mathcal M$-cellular, and then so is the homotopy fiber of the map to the homotopy cofiber. However, this rationalized sphere, which is a third Postnikov section with only two non-trivial homotopy groups (two copies of $\mathbb Q$), is the homological localization $L_{H \mathbb Q} S^2$, hence also the associated rational plus construction. Altogether, we have shown that the homotopy fiber of the rationalization map $S^2 \rightarrow (S^2)_{(0)}$ is $\mathcal M$-cellular.
It coincides with $\overline{L}_{H \mathbb Q} S^2$ and the $H\mathbb Q$-acyclization of $S^2$, which is therefore in particular the $\mathcal M$-cellularization of $S^2$.
We end this introductory example by observing that this space has two copies of $\mathbb Q/\mathbb Z$ in its first two homotopy groups, while the higher homotopy groups are isomorphic to the torsion part of the homotopy groups of $S^2$. 

Alternatively, one could use Bousfield's computation in \cite[Theorem 7.5]{Bousfield1997} (note the small typo there where the $n$-index is switched to $n+1$) so as to identify these homotopy groups and compare them with those of the cellularization. 
\end{example}

Let us move now to the general case and prove that $\mathcal M$-cellularization always coincides with $H\Z[J^{-1}]$-acyclization. It is a general fact that $\mathcal{M}$-cellular spaces are $\mathcal{M}$-acyclic, but the other inclusion is not a tautology.

\begin{theorem}
\label{thm:acyclic=cellular}
The three following classes coincide: $\mathcal M$-cellular spaces, $\mathcal M$-acyclic spaces, and $H\mathbb Z[J^{-1}]$-acyclic spaces. In particular the functor $cell_{\mathcal M}$ is $H\mathbb Z[J^{-1}]$-acyclization.
\end{theorem}
\begin{proof}
Any $\mathcal M$-cellular space is $H\Z[J^{-1}]$-acyclic as $\mathcal M$ itself is $H\mathbb Z[J^{-1}]$-acyclic by \cref{lem:Moore space}. Conversely, if $X$ is $H\mathbb Z[J^{-1}]$-acyclic, its fundamental group is $H\mathbb Z[J^{-1}]$-perfect.
Hence, an application of \cref{prop:effect on pi1} 
exhibits then the homotopy cofiber $X'$ of $\vee \mathcal M \rightarrow X$ as a simply-connected and $H\mathbb Z[J^{-1}]$-acyclic space. It is therefore $J$-torsion, and such a space is killed by $\Sigma \mathcal M$, which shows that $X$ is cellular by Chach\'olski's fibration, \cite[Theorem~20.3]{MR1408539}.

The third closed class, that of $\mathcal M$-acyclic spaces $\overline{\mathcal C(\mathcal M)}$, 
contains $\mathcal{M}$-cellular spaces by definition, and is contained in the class of $H\Z[J^{-1}]$-acyclic spaces because this property is preserved under all closure properties (equivalences, pointed homotopy colimits, and extensions by fibrations). 

The equivalence between cellularization and acyclization now follows directly since both functors are coreflections on the same closed class.
\end{proof}

\begin{remark}
\label{rem:pedestrian}
For a more pedestrian argument, closer in spirit to \cref{ex:sphere} of the sphere, we can also use Chach\'olski's fibration in \cref{thm:Chacholski} applied to the acyclization map $A_{H\mathbb Z[J^{-1}]} X \rightarrow X$ for an arbitrary connected space $X$. The source is $H\mathbb Z[J^{-1}]$-acyclic, hence $\mathcal M$-cellular, and it has the property that any map $\mathcal M \rightarrow X$ factors through it. Therefore the cellularization $cell_{\mathcal M} X$ can be computed
as the homotopy fiber of the $\Sigma \mathcal M$-nullification of the homotopy cofiber $X' = \mathrm{Cof}(A_{H\mathbb Z[J^-1]}X \rightarrow X)$.

By Ganea's fiber-cofiber trick, \cite[Theorem~1.1]{MR179791}, we can compute the homotopy fiber of the induced map $X' \rightarrow P_{\mathcal M} X$ and identify it as the join $\Omega P_{\mathcal M} X \ast A_{H\mathbb Z[J^-1]} X$, which is equivalent to a suspension $\Sigma \Omega P_{\mathcal M} X \wedge A_{H\mathbb Z[J^-1]} X$. However, since the nullification of a connected space is connected, the suspension of its loop space is connected, i.e., killed by $S^1$, so that this homotopy fiber is killed by $S^1 \wedge A_{H\mathbb Z[J^-1]} X$, hence by $S^1 \wedge \mathcal M = \Sigma \mathcal M$.

In other words, the map $X' \rightarrow P_{\mathcal M} X$ is a $P_{\Sigma \mathcal M}$-equivalence. This tells us that the \mbox{$\mathcal M$-cellularization} $cell_{\mathcal M} X$, which can be constructed as the homotopy fiber of the composite $X \rightarrow X' \rightarrow P_{\Sigma \mathcal M} X'$, coincides with $\overline{P}_{\mathcal M} X$.
\end{remark}

\section{Is the acyclization fibration a cofiber sequence?}
\label{Section:cofibrations}
In the integral case, the acyclization-plus construction fibration is always a cofibration, and this means that the plus construction $X^+$ can be seen as the homotopy cofiber of the acyclization map $AX \rightarrow X$. This phenomenon has been studied by Alonso in \cite{Alonso1} and \cite{Alonso2}, see also Wenhuai and Zai-si's generalization \cite{MR1358818}, and Raptis' interesting modern point of view in \cite{MR3987558}. For rational homology, this is not the case in general, as we illustrate by an elementary example. 

\begin{example}
\label{ex:not a cofibration}
We work with the rational plus construction. Let us consider the product $X= \mathbb R P^2 \times S^1$. The projective plane is a torsion Moore space, hence $H\mathbb Q$-acyclic, whereas the circle is $\mathcal M$-local: its fundamental group does not contain torsion elements up to commutators and there are no non-trivial maps $\Sigma^k \mathcal M \rightarrow S^1$. Therefore, the acyclization fiber sequence for this space is
\[
A_{H \mathbb Q} (X) = \mathbb R P^2 \rightarrow X \rightarrow S^1 = X^{+H \mathbb Q}
\]
This is not a cofiber sequence because the homotopy cofiber $X/\mathbb R P^2$ is the half-smash product $\mathbb R P^2 \ltimes S^1$. One way to see that this is not a circle is to compute the homotopy fiber of the comparison map $\mathbb R P^2 \ltimes S^1 \rightarrow S^1$, which is the join 
\[
\Omega S^1 * \mathbb RP^2 \simeq \Sigma \mathbb Z \wedge \mathbb RP^2 \simeq \bigvee_{-\infty}^{\infty} \Sigma \mathbb RP^2  
\]
\end{example}

The previous example is probably too simple to understand the general situation, but we chose to highlight the comparison map because it will play a key role in the next results.
\begin{lemma}
\label{lem:comparison map}
Let $X$ be a connected space and $A_{H\mathbb Z[J^{-1}]} X \rightarrow X \rightarrow X^{+  H\mathbb Z[J^{-1}]}$ be the associated acyclization fiber sequence. The homotopy fiber of the comparison map 
\[
X/A_{H\mathbb Z[J^{-1}]} X \rightarrow X^{+  H\mathbb Z[J^{-1}]}
\]
is equivalent to $\Sigma A_{H\mathbb Z[J^{-1}]} X \wedge\Omega X^{+  H\mathbb Z[J^{-1}]}$.
\end{lemma}

\begin{proof}
This is another application of Ganea's fiber-cofiber trick; the homotopy fiber of the comparison map is the join $A_{H\mathbb Z[J^{-1}]} X * \Omega X^{+  H\mathbb Z[J^{-1}]}$, \cite[Theorem~1.1]{MR179791}.
\end{proof}

We wish to understand when the comparison map is an equivalence, or in other words, when the homotopy fiber we just computed is contractible. There are two very different situations now, depending on the connectivity of the loop space. 

\begin{prop}
\label{prop:it is a cofibration}
Let $X$ be a connected space. If the $H \mathbb Z[J^{-1}]$-plus construction $X^{+  H\mathbb Z[J^{-1}]}$ is simply connected, then the acyclization fiber sequence  $A_{H\mathbb Z[J^{-1}]} X \rightarrow X \rightarrow X^{+  H\mathbb Z[J^{-1}]}$ is also a cofiber sequence. 
\end{prop}

\begin{proof}
The suspension $\Sigma \mathcal M$ splits as a wedge of $J$-torsion Moore spaces, see \cref{rem:Moore}. We can therefore use Bousfield's computation \cite[Theorem~7.5]{Bousfield1997} and infer that the homotopy groups of the simply connected and $\mathcal M$-local space $X^{+  H\mathbb Z[J^{-1}]}$ are uniquely $J$-divisible (the plus construction coincides in fact with the $H \mathbb Z[J^{-1}]$-homological localization here). Therefore so are the homotopy groups of its loop space, and the homology groups as well.

On the other hand the suspension of $A_{H\mathbb Z[J^{-1}]} X$ is a simply connected and $H\mathbb Z [J^{-1}]$-acyclic space, hence $J$-torsion. We are now in position to apply Alonso's criterion \cite[Corollary~2.10]{Alonso1} and conclude that the acyclization fiber sequence is also a cofiber sequence. This follows from the fact that the smash product in \cref{lem:comparison map} of a $J$-torsion space with a $J$-divisible one is contractible.
\end{proof}

This result applies in particular to any simply connected space $X$, such as the $2$-sphere we looked at in \cref{ex:sphere}. We move on to the other case, when the plus construction is not simply connected.

\begin{prop}
\label{prop:it is not a cofibration}
Let $X$ be a connected space and assume that the $H \mathbb Z[J^{-1}]$-plus construction $X^{+  H\mathbb Z[J^{-1}]}$ is not simply connected. Then the acyclization fiber sequence 
\[
{A_{H\mathbb Z[J^{-1}]} X \rightarrow X \rightarrow X^{+  H\mathbb Z[J^{-1}]}}
\]
is a cofiber sequence if and only if $A_{H\mathbb Z[J^{-1}]} X$ is $H \mathbb Z$-acyclic. 
\end{prop}

\begin{proof}
When the plus construction is not simply connected, its fundamental group contains at least two elements, and its loop space consists of at least two connected components. The $0$-sphere $S^0$ is thus a retract of $\Omega X^{+  H\mathbb Z[J^{-1}]}$ and so, the smash product ${\Sigma A_{H\mathbb Z[J^{-1}]} X \wedge S^0 \simeq \Sigma A_{H\mathbb Z[J^{-1}]} X}$ is a retract of $\Sigma A_{H\mathbb Z[J^{-1}]} X \wedge\Omega X^{+  H\mathbb Z[J^{-1}]}$, the homotopy fiber of the comparison map we computed in \cref{lem:comparison map}.

We conclude now by analyzing the two possible cases. If $A_{H\mathbb Z[J^{-1}]} X$ is not $H \mathbb Z$-acyclic, its suspension is not contractible, which implies that the homotopy fiber of the comparison map is not contractible either by the previous computation. In this case, the acyclization fiber sequence cannot be a cofiber sequence. In the second case, when $A_{H\mathbb Z[J^{-1}]} X$ is $H \mathbb Z$-acyclic, its suspension is contractible, hence so is the homotopy fiber of the comparison map.
\end{proof}

These two propositions provide a complete answer to the question of when the acyclization fiber sequence is also a cofiber sequence.

\begin{theorem}
\label{thm:cofibration}
Let $X$ be a connected space. Then the acyclization fiber sequence 
\[
A_{H\mathbb Z[J^{-1}]} X \rightarrow X \rightarrow X^{+ H\mathbb Z[J^{-1}]}
\]
is a cofiber sequence if and only if one of the two following conditions holds:
\begin{itemize}
    \item[(a)] The $H \mathbb Z[J^{-1}]$-plus construction $X^{+  H\mathbb Z[J^{-1}]}$ is simply connected;
    \item[(b)] the $H\mathbb Z[J^{-1}]$-acyclization $A_{H\mathbb Z[J^{-1}]} X$ is $H \mathbb Z$-acyclic. 
\end{itemize}
\end{theorem}

We conclude this article with an open question.

\begin{remark}
\label{rem:initial and acyclic}
By design, the $H \Z[J^{-1}]$-plus construction is a nullification functor, so the nullification map $X \rightarrow X^{+  \Z[J^{-1}]}$ is initial among all maps to $\mathcal M$-local spaces and terminal among $\mathcal M$-local equivalences out of $X$.

The integral plus construction enjoys yet another universal property. It is initial among all acyclic maps inducing the quotient by the maximal perfect subgroup on the fundamental groups. In our setting, this is no longer true because the nullification functor we have designed is not the most economic way to achieve this. In fact, it could be that the Broto-Levi-Oliver construction from \cite{Broto2021} is closer to enjoying such a universal property as it is somewhat constructed by attaching as few cells as possible to kill the desired subgroup of $\pi_1 X$. We do not know, however, how to formalize this since this construction is not functorial.
\end{remark}
{\sloppy\printbibliography}

@article {Alonso1,
    AUTHOR = {Alonso, J. M.},
     TITLE = {Fibrations that are cofibrations},
   JOURNAL = {Proc. Amer. Math. Soc.},
  FJOURNAL = {Proceedings of the American Mathematical Society},
    VOLUME = {87},
      YEAR = {1983},
    NUMBER = {4},
     PAGES = {749--753},
      ISSN = {0002-9939,1088-6826},
   MRCLASS = {55P60 (55P05 55R05)},
  MRNUMBER = {687656},
MRREVIEWER = {R.\ M.\ Vogt},
       DOI = {10.2307/2043374},
       URL = {https://doi.org/10.2307/2043374},
}

@article {Alonso2,
    AUTHOR = {Alonso, J. M.},
     TITLE = {Fibrations that are cofibrations. {II}},
   JOURNAL = {Proc. Amer. Math. Soc.},
  FJOURNAL = {Proceedings of the American Mathematical Society},
    VOLUME = {105},
      YEAR = {1989},
    NUMBER = {2},
     PAGES = {486--490},
      ISSN = {0002-9939,1088-6826},
   MRCLASS = {55P05 (55P60 55R05)},
  MRNUMBER = {977927},
MRREVIEWER = {R.\ M.\ Vogt},
       DOI = {10.2307/2046968},
       URL = {https://doi.org/10.2307/2046968},
}

@Article{Berrick1999,
  author     = {Berrick, A. J. and Casacuberta, C.},
  journal    = {Topology},
  title      = {A universal space for plus-constructions},
  year       = {1999},
  issn       = {0040-9383},
  number     = {3},
  pages      = {467--477},
  volume     = {38},
  doi        = {10.1016/S0040-9383(97)00073-6},
  fjournal   = {Topology. An International Journal of Mathematics},
  mrclass    = {55P60 (19D06 20E26)},
  mrnumber   = {1670384},
  mrreviewer = {J. P. C. Greenlees},
  url        = {https://doi.org/10.1016/S0040-9383(97)00073-6},
}

@article {Bousfield94,
    AUTHOR = {Bousfield, A. K.},
     TITLE = {Localization and periodicity in unstable homotopy theory},
   JOURNAL = {J. Amer. Math. Soc.},
  FJOURNAL = {Journal of the American Mathematical Society},
    VOLUME = {7},
      YEAR = {1994},
    NUMBER = {4},
     PAGES = {831--873},
      ISSN = {0894-0347},
   MRCLASS = {55P60 (55N15 55N20 55Q05 55U10)},
  MRNUMBER = {1257059},
MRREVIEWER = {N. J. Kuhn},
       DOI = {10.2307/2152734},
       URL = {https://doi.org/10.2307/2152734},
}

@article {Bousfield75,
    AUTHOR = {Bousfield, A. K.},
     TITLE = {The localization of spaces with respect to homology},
   JOURNAL = {Topology},
  FJOURNAL = {Topology. An International Journal of Mathematics},
    VOLUME = {14},
      YEAR = {1975},
     PAGES = {133--150},
      ISSN = {0040-9383},
   MRCLASS = {55D05},
  MRNUMBER = {380779},
MRREVIEWER = {M. Mimura},
       DOI = {10.1016/0040-9383(75)90023-3},
       URL = {https://doi.org/10.1016/0040-9383(75)90023-3},
  sort = {Bou}, }

@article {Bousfield1997,
    AUTHOR = {Bousfield, A. K.},
     TITLE = {Homotopical localizations of spaces},
   JOURNAL = {Amer. J. Math.},
  FJOURNAL = {American Journal of Mathematics},
    VOLUME = {119},
      YEAR = {1997},
    NUMBER = {6},
     PAGES = {1321--1354},
      ISSN = {0002-9327},
   MRCLASS = {55P60},
  MRNUMBER = {1481817},
MRREVIEWER = {Dirk Scevenels},
URL={http://muse.jhu.edu/journals/american_journal_of_mathematics/v119/119.6bousfield.pdf}
}

@Article{Broto2021,
  author   = {Broto, C. and Levi, R. and Oliver, B.},
  journal  = {Ann. K-Theory},
  title    = {Loop space homology of a small category},
  year     = {2021},
  issn     = {2379-1683},
  number   = {3},
  pages    = {425--480},
  volume   = {6},
  doi      = {10.2140/akt.2021.6.425},
  fjournal = {Annals of K-Theory},
  mrclass  = {55R35 (20J99 55R40)},
  mrnumber = {4310325},
  url      = {https://doi.org/10.2140/akt.2021.6.425},
  sort = {Bro}, 
}

@incollection {MR1290581,
    AUTHOR = {Casacuberta, C.},
     TITLE = {Recent advances in unstable localization},
 BOOKTITLE = {The {H}ilton {S}ymposium 1993 ({M}ontreal, {PQ})},
    SERIES = {CRM Proc. Lecture Notes},
    VOLUME = {6},
     PAGES = {1--22},
 PUBLISHER = {Amer. Math. Soc., Providence, RI},
      YEAR = {1994},
   MRCLASS = {55P60 (18A40)},
  MRNUMBER = {1290581},
MRREVIEWER = {Donald M. Davis},
       DOI = {10.1090/crmp/006/01},
       URL = {https://doi.org/10.1090/crmp/006/01},
}

@InCollection{Casacuberta1999,
  author     = {Casacuberta, C. and Rodr\'{\i}guez, J. L. and Scevenels, D.},
  booktitle  = {Groups {S}t. {A}ndrews 1997 in {B}ath, {I}},
  publisher  = {Cambridge Univ. Press, Cambridge},
  title      = {Singly generated radicals associated with varieties of groups},
  year       = {1999},
  pages      = {202--210},
  series     = {London Math. Soc. Lecture Note Ser.},
  volume     = {260},
  mrclass    = {20E10 (20J15)},
  mrnumber   = {1676617},
  mrreviewer = {Thomas A. Fournelle},
}

@article {Cha04,
    AUTHOR = {Chach\'olski, W. and Parent, P. E. and Stanley,
              D.},
     TITLE = {Cellular generators},
   JOURNAL = {Proc. Amer. Math. Soc.},
  FJOURNAL = {Proceedings of the American Mathematical Society},
    VOLUME = {132},
      YEAR = {2004},
    NUMBER = {11},
     PAGES = {3397--3409},
      ISSN = {0002-9939,1088-6826},
   MRCLASS = {55Q05},
  MRNUMBER = {2073317},
MRREVIEWER = {Maria\ Cristina\ Costoya-Ramos},
       DOI = {10.1090/S0002-9939-04-07346-0},
       URL = {https://doi.org/10.1090/S0002-9939-04-07346-0},
}

@article {MR4381283,
    AUTHOR = {Craig, W. and Linnell, P. A.},
     TITLE = {Unique product groups and congruence subgroups},
   JOURNAL = {J. Algebra Appl.},
  FJOURNAL = {Journal of Algebra and its Applications},
    VOLUME = {21},
      YEAR = {2022},
    NUMBER = {2},
     PAGES = {Paper No. 2250025, 9},
      ISSN = {0219-4988,1793-6829},
   MRCLASS = {20E18 (16S34 16U60)},
  MRNUMBER = {4381283},
MRREVIEWER = {Benjamin\ Klopsch},
       DOI = {10.1142/S0219498822500256},
       URL = {https://doi.org/10.1142/S0219498822500256},
}

@Book{Farjoun1996,
  author     = {Farjoun, E. D.},
  publisher  = {Springer-Verlag, Berlin},
  title      = {Cellular spaces, null spaces and homotopy localization},
  year       = {1996},
  isbn       = {3-540-60604-1},
  series     = {Lecture Notes in Mathematics},
  volume     = {1622},
  doi        = {10.1007/BFb0094429},
  mrclass    = {55P60 (55-02 55P65)},
  mrnumber   = {1392221},
  mrreviewer = {Carles Casacuberta},
  pages      = {xiv+199},
  url        = {https://doi.org/10.1007/BFb0094429},
}

@incollection {MR1796125,
    AUTHOR = {Casacuberta, C.},
     TITLE = {On structures preserved by idempotent transformations of
              groups and homotopy types},
 BOOKTITLE = {Crystallographic groups and their generalizations ({K}ortrijk,
              1999)},
    SERIES = {Contemp. Math.},
    VOLUME = {262},
     PAGES = {39--68},
 PUBLISHER = {Amer. Math. Soc., Providence, RI},
      YEAR = {2000},
      ISBN = {0-8218-2001-X},
   MRCLASS = {55P60 (18A40)},
  MRNUMBER = {1796125},
MRREVIEWER = {Donald\ M.\ Davis},
       DOI = {10.1090/conm/262/04167},
       URL = {https://doi.org/10.1090/conm/262/04167},
}

@article {Flores21,
    AUTHOR = {Flores, R. and Rodr\'iguez, J. L.},
     TITLE = {Generators and closed classes of groups},
   JOURNAL = {Publ. Mat.},
  FJOURNAL = {Publicacions Matem\`atiques},
    VOLUME = {65},
      YEAR = {2021},
    NUMBER = {2},
     PAGES = {431--457},
      ISSN = {0214-1493,2014-4350},
   MRCLASS = {20E34 (20E22 20F05 20J15 55P60)},
  MRNUMBER = {4278755},
MRREVIEWER = {Martino\ Garonzi},
       DOI = {10.5565/publmat6522102},
       URL = {https://doi.org/10.5565/publmat6522102},
}

@article {MR179791,
    AUTHOR = {Ganea, T.},
     TITLE = {A generalization of the homology and homotopy suspension},
   JOURNAL = {Comment. Math. Helv.},
  FJOURNAL = {Commentarii Mathematici Helvetici},
    VOLUME = {39},
      YEAR = {1965},
     PAGES = {295--322},
      ISSN = {0010-2571},
   MRCLASS = {55.40},
  MRNUMBER = {179791},
MRREVIEWER = {P. J. Hilton},
       DOI = {10.1007/BF02566956},
       URL = {https://doi.org/10.1007/BF02566956},
}

@article {Kervaire69,
    AUTHOR = {Kervaire, M. A.},
     TITLE = {Smooth homology spheres and their fundamental groups},
   JOURNAL = {Trans. Amer. Math. Soc.},
  FJOURNAL = {Transactions of the American Mathematical Society},
    VOLUME = {144},
      YEAR = {1969},
     PAGES = {67--72},
      ISSN = {0002-9947,1088-6850},
   MRCLASS = {57.10},
  MRNUMBER = {253347},
MRREVIEWER = {A.\ Liulevicius},
       DOI = {10.2307/1995269},
       URL = {https://doi.org/10.2307/1995269},
}

@Article{Mislin2001,
  author     = {Mislin, G. and Peschke, G.},
  journal    = {Trans. Amer. Math. Soc.},
  title      = {Central extensions and generalized plus-constructions},
  year       = {2001},
  issn       = {0002-9947},
  number     = {2},
  pages      = {585--608},
  volume     = {353},
  doi        = {10.1090/S0002-9947-00-02727-6},
  fjournal   = {Transactions of the American Mathematical Society},
  mrclass    = {19D06 (55P60 55Q15)},
  mrnumber   = {1804509},
  mrreviewer = {Ross Staffeldt},
  url        = {https://doi.org/10.1090/S0002-9947-00-02727-6},
}

@article {MR1408539,
    AUTHOR = {Chach\'{o}lski, W.},
     TITLE = {On the functors {$CW_A$} and {$P_A$}},
   JOURNAL = {Duke Math. J.},
  FJOURNAL = {Duke Mathematical Journal},
    VOLUME = {84},
      YEAR = {1996},
    NUMBER = {3},
     PAGES = {599--631},
      ISSN = {0012-7094},
   MRCLASS = {55P99 (55U40)},
  MRNUMBER = {1408539},
MRREVIEWER = {Donald M. Davis},
       DOI = {10.1215/S0012-7094-96-08419-7},
       URL = {https://doi.org/10.1215/S0012-7094-96-08419-7},
}

@article {MR4602845,
    AUTHOR = {Carri\'{o}n Santiago, G. and Scherer, J.},
     TITLE = {Relative plus constructions},
   JOURNAL = {Expo. Math.},
  FJOURNAL = {Expositiones Mathematicae},
    VOLUME = {41},
      YEAR = {2023},
    NUMBER = {2},
     PAGES = {316--332},
      ISSN = {0723-0869,1878-0792},
   MRCLASS = {55P60 (19D06 20F14 20F34 55N25)},
  MRNUMBER = {4602845},
       DOI = {10.1016/j.exmath.2023.03.001},
       URL = {https://doi.org/10.1016/j.exmath.2023.03.001},
}

@incollection {MR1320986,
    AUTHOR = {Casacuberta, C.},
     TITLE = {Anderson localization from a modern point of view},
 BOOKTITLE = {The \v Cech centennial ({B}oston, {MA}, 1993)},
    SERIES = {Contemp. Math.},
    VOLUME = {181},
     PAGES = {35--44},
 PUBLISHER = {Amer. Math. Soc., Providence, RI},
      YEAR = {1995},
      ISBN = {0-8218-0296-8},
   MRCLASS = {55P60 (55N25 55Q05)},
  MRNUMBER = {1320986},
MRREVIEWER = {George\ Peschke},
       DOI = {10.1090/conm/181/02028},
       URL = {https://doi.org/10.1090/conm/181/02028},
}

@article {MR1318881,
    AUTHOR = {Farjoun, E. D. and Smith, J. H.},
     TITLE = {Homotopy localization nearly preserves fibrations},
   JOURNAL = {Topology},
  FJOURNAL = {Topology. An International Journal of Mathematics},
    VOLUME = {34},
      YEAR = {1995},
    NUMBER = {2},
     PAGES = {359--375},
      ISSN = {0040-9383},
   MRCLASS = {55P60 (55N20 55P20)},
  MRNUMBER = {1318881},
MRREVIEWER = {Carles\ Casacuberta},
       DOI = {10.1016/0040-9383(94)00025-G},
       URL = {https://doi.org/10.1016/0040-9383(94)00025-G},
}

@incollection {MR338129,
    AUTHOR = {Quillen, D.},
     TITLE = {Higher algebraic {$K$}-theory. {I}},
 BOOKTITLE = {Algebraic {$K$}-theory, {I}: {H}igher {$K$}-theories ({P}roc.
              {C}onf., {B}attelle {M}emorial {I}nst., {S}eattle, {W}ash.,
              1972)},
    SERIES = {Lecture Notes in Math.},
    VOLUME = {Vol. 341},
     PAGES = {85--147},
 PUBLISHER = {Springer, Berlin-New York},
      YEAR = {1973},
   MRCLASS = {18F25},
  MRNUMBER = {338129},
MRREVIEWER = {Stephen\ M.\ Gersten},
}

@article {MR4379387,
    AUTHOR = {Popko, J. and Szczepa\'nski, A.},
     TITLE = {Properties of the combinatorial {H}antzsche-{W}endt groups},
   JOURNAL = {Topology Appl.},
  FJOURNAL = {Topology and its Applications},
    VOLUME = {310},
      YEAR = {2022},
     PAGES = {Paper No. 108037, 15},
      ISSN = {0166-8641,1879-3207},
   MRCLASS = {20J05 (20E06 20H15 55R20)},
  MRNUMBER = {4379387},
MRREVIEWER = {Mahender\ Singh},
       DOI = {10.1016/j.topol.2022.108037},
       URL = {https://doi.org/10.1016/j.topol.2022.108037},
}

@article {MR3987558,
    AUTHOR = {Raptis, G.},
     TITLE = {Some characterizations of acyclic maps},
   JOURNAL = {J. Homotopy Relat. Struct.},
  FJOURNAL = {Journal of Homotopy and Related Structures},
    VOLUME = {14},
      YEAR = {2019},
    NUMBER = {3},
     PAGES = {773--785},
      ISSN = {2193-8407,1512-2891},
   MRCLASS = {55U35 (18B25)},
  MRNUMBER = {3987558},
MRREVIEWER = {Thomas\ H\"uttemann},
       DOI = {10.1007/s40062-019-00231-6},
       URL = {https://doi.org/10.1007/s40062-019-00231-6},
}

@incollection {MR1477181,
    AUTHOR = {Strojnowski, A.},
     TITLE = {On torsion-free groups with trivial center},
 BOOKTITLE = {Infinite groups 1994 ({R}avello)},
     PAGES = {241--246},
 PUBLISHER = {de Gruyter, Berlin},
      YEAR = {1996},
      ISBN = {3-11-014332-1},
   MRCLASS = {20F34},
  MRNUMBER = {1477181},
MRREVIEWER = {Juan\ Pablo\ Rossetti},
}

@article {MR1640095,
    AUTHOR = {Tai, J.-Y.},
     TITLE = {Generalized plus-constructions and fundamental groups},
   JOURNAL = {J. Pure Appl. Algebra},
  FJOURNAL = {Journal of Pure and Applied Algebra},
    VOLUME = {132},
      YEAR = {1998},
    NUMBER = {2},
     PAGES = {207--220},
      ISSN = {0022-4049,1873-1376},
   MRCLASS = {55P60 (19D06 55N20)},
  MRNUMBER = {1640095},
MRREVIEWER = {V.\ P.\ Snaith},
       DOI = {10.1016/S0022-4049(97)00104-7},
       URL = {https://doi.org/10.1016/S0022-4049(97)00104-7},
}

@article {MR1358818,
    AUTHOR = {Wenhuai, S. and Zai-Si, Z.},
     TITLE = {On fibrations that are cofibrations},
   JOURNAL = {Topology Appl.},
  FJOURNAL = {Topology and its Applications},
    VOLUME = {66},
      YEAR = {1995},
    NUMBER = {2},
     PAGES = {159--169},
      ISSN = {0166-8641,1879-3207},
   MRCLASS = {55P60 (55P05 55R05)},
  MRNUMBER = {1358818},
MRREVIEWER = {Zafer\ Mahmud},
       DOI = {10.1016/0166-8641(95)00025-C},
       URL = {https://doi.org/10.1016/0166-8641(95)00025-C},
}

@article {MR4334981,
    AUTHOR = {Gardam, G.},
     TITLE = {A counterexample to the unit conjecture for group rings},
   JOURNAL = {Ann. of Math. (2)},
  FJOURNAL = {Annals of Mathematics. Second Series},
    VOLUME = {194},
      YEAR = {2021},
    NUMBER = {3},
     PAGES = {967--979},
      ISSN = {0003-486X,1939-8980},
   MRCLASS = {20C07},
  MRNUMBER = {4334981},
MRREVIEWER = {E.\ Formanek},
       DOI = {10.4007/annals.2021.194.3.9},
       URL = {https://doi.org/10.4007/annals.2021.194.3.9},
}

@article {MR4601076,
    AUTHOR = {Ivanov, S. O. and Mikhailov, R.},
     TITLE = {H{R}-length of a free group via polynomial functors},
   JOURNAL = {J. Pure Appl. Algebra},
  FJOURNAL = {Journal of Pure and Applied Algebra},
    VOLUME = {227},
      YEAR = {2023},
    NUMBER = {12},
     PAGES = {Paper No. 107443, 13},
      ISSN = {0022-4049,1873-1376},
   MRCLASS = {18N55 (19D55 20J06)},
  MRNUMBER = {4601076},
       DOI = {10.1016/ j.jpaa.2023.107443},
       URL = {https://doi.org/10.1016/j.jpaa.2023.107443},
}

@article {MR447375,
    AUTHOR = {Bousfield, A. K.},
     TITLE = {Homological localization towers for groups and {$\Pi
              $}-modules},
   JOURNAL = {Mem. Amer. Math. Soc.},
  FJOURNAL = {Memoirs of the American Mathematical Society},
    VOLUME = {10},
      YEAR = {1977},
    NUMBER = {186},
     PAGES = {vii+68},
      ISSN = {0065-9266,1947-6221},
   MRCLASS = {18H10 (55B25)},
MRREVIEWER = {Harold\ Hastings},
       DOI = {10.1090/memo/0186},
       URL = {https://doi.org/10.1090/memo/0186},
}

@article {MR3009739,
    AUTHOR = {Ye, Shengkui},
     TITLE = {A unified approach to the plus-construction, {B}ousfield
              localization, {M}oore spaces and zero-in-the-spectrum
              examples},
   JOURNAL = {Israel J. Math.},
  FJOURNAL = {Israel Journal of Mathematics},
    VOLUME = {192},
      YEAR = {2012},
    NUMBER = {2},
     PAGES = {699--717},
      ISSN = {0021-2172,1565-8511},
   MRCLASS = {19D06 (55T20)},
  MRNUMBER = {3009739},
MRREVIEWER = {Stanis\l aw\ Betley},
       DOI = {10.1007/s11856-012-0051-y},
       URL = {https://doi.org/10.1007/s11856-012-0051-y},
}
\end{document}